\documentclass[12pt,reqno]{amsart}

\usepackage{amsfonts,amsmath,amsthm}
\usepackage{amssymb,epsfig}
\usepackage{enumerate} 

\usepackage{color} 
\definecolor{vert}{rgb}{0,0.6,0}

\usepackage[text={425pt,650pt},centering]{geometry}
\usepackage{graphicx}
\usepackage{epsfig}
\usepackage{tikz}
\usepackage{caption}
\usepackage{color} 
\definecolor{vert}{rgb}{0,0.6,0}

\numberwithin{figure}{section}

\theoremstyle{plain}
\newtheorem{thm}{Theorem}[section]

\newtheorem{defn}{Definition}
\newtheorem{quest}{Question}
\newtheorem{conj}{Conjecture}

\newtheorem{lem}[thm]{Lemma}
\newtheorem{cor}[thm]{Corollary}
\newtheorem{prop}[thm]{Proposition}
\theoremstyle{remark}
\newtheorem{rem}{\bf{Remark}}
\numberwithin{equation}{section}

\newcommand{\N}{\mathbb{N}}

\newcommand{\R}{\mathbb{R}}

\newcommand{\T}{\mathbb{T}}
\newcommand{\Z}{\mathbb{Z}}

\newcommand{\BUC}{{\rm BUC\,}}

\newcommand{\ep}{\varepsilon}

\newcommand{\lam}{\lambda}

\newcommand{\ol}{\overline}

\begin{document}
\title[Inverse problems in homogenization]{A rigidity result for effective Hamiltonians with $3$-mode periodic potentials}

\author{Hung V. Tran}
\address[Hung V. Tran]
{
Department of Mathematics, 
University of Wisconsin Madison, Van Vleck hall, 480 Lincoln drive, Madison, WI 53706, USA}
\email{hung@math.wisc.edu}

\author{Yifeng Yu}
\address[Yifeng Yu]
{
Department of Mathematics, 
University of California, Irvine, 410G Rowland Hall, Irvine, CA 92697, USA}
\email{yyu1@math.uci.edu}

\thanks{
The work of HT is partially supported by NSF grants DMS-1615944 and  DMS-1664424,
the work of YY is partially supported by NSF CAREER award \#1151919.
}

\date{}


\maketitle 

\begin{abstract}
We continue studying an inverse problem in the theory of periodic homogenization of Hamilton-Jacobi equations proposed in \cite{LTY}.
Let $V_1, V_2 \in C(\R^n)$ be two given potentials which are $\Z^n$-periodic,
and $\ol{H}_1, \ol{H}_2$ be the effective Hamiltonians associated with the Hamiltonians $\frac{1}{2}|p|^2 + V_1$, $\frac{1}{2}|p|^2+V_2$, respectively.

A main result in this paper is that,  if the dimension $n=2$ and each of $V_1, V_2$ contains exactly $3$ mutually non-parallel Fourier modes, then
$$
\overline H_1\equiv \overline H_2   \quad \iff \quad V_1(x)=V_2\left({x\over c}+x_0\right) \quad \text{ for all } x \in \T^2 = \R^2/\Z^2,
$$
for some $c\in \mathbb{Q} \setminus\{0\}$ and $x_0 \in \T^2$.
When $n\geq 3$, the scenario is slightly more subtle, and a complete description is  provided for any dimension.  
These resolve partially the conjecture stated in \cite{LTY}.
Some other related results and open problems are also discussed.

\end{abstract}

\section{Introduction}

\subsection{Periodic homogenization and the inverse problem}
We first describe the theory of periodic homogenization of Hamilton-Jacobi equations.
For each length scale $\ep>0$, let $u^\ep \in C(\R^n \times [0,\infty))$ be the viscosity solution to
\begin{equation}\label{HJ-ep}
\begin{cases}
u^\ep_t + H(Du^\ep) +V\left(\frac{x}{\ep}\right)=0 \quad &\text{ in } \R^n \times (0,\infty),\\
u^\ep(x,0)=g(x) \quad &\text{ on } \R^n.
\end{cases}
\end{equation}
Here, the Hamiltonian $H(p) -V(x)$ is of separable form with $H \in C(\R^n)$, which is coercive (i.e., $\lim_{|p| \to \infty} H(p)=+\infty$),
and $V \in C(\R^n)$, which is $\Z^n$-periodic.
The initial data $g\in \BUC(\R^n)$, the set of bounded, uniformly continuous functions on $\R^n$.

It was shown in \cite{LPV} that, in the limit as the length scale $\ep$ tends to zero,
 $u^\ep$ converges to $u$ locally uniformly on $\R^n \times [0,\infty)$,
and $u$ solves the effective equation
\begin{equation}\label{HJ-hom}
\begin{cases}
u_t +\ol{H}(Du)=0 \quad &\text{ in } \R^n \times (0,\infty),\\
u(x,0)=g(x) \quad &\text{ on } \R^n.
\end{cases}
\end{equation}
The effective Hamiltonian $\ol{H} \in C(\R^n)$ is determined in a nonlinear way by $H$ and $V$ through the cell problems as following.  
For each $p \in \R^n$, it was derived in \cite{LPV} that there exists a unique constant $c \in \R$ such that the following cell problem has a continuous viscosity solution
\begin{equation} \label{E-p}
H(p+Dv) + V(x) = c \quad \text{ in } \T^n,
\end{equation}
where $\T^n$ is the $n$-dimensional flat torus $\R^n/\Z^n$.
We then denote by $\ol{H}(p):=c$.

During past decades,  there have been tremendous progress and vast literature about the validity of homogenization and the well-posedness of cell problems in various generalized settings.  Nevertheless,  understanding theoretically how $\overline H$ depends on the potential $V$ remains a very challenging and still largely open problem even for the most basic  case $H(p)={1\over 2}|p|^2$.   
For a smooth periodic potential $V$,  a deep result  in \cite{B2} asserts that when $n=2$, each non-minimum level curve of $\overline H$ associated with ${1\over 2}|p|^2-V$ must contain line segments unless $V$ is constant. Its proof relies on delicate analysis based on detailed structure of Aubry-Mather sets in two dimensions and a rigidity result in Riemannian geometry (the Hopf conjecture).  Besides, due to the highly nonlinear nature of the problem,  efficient numerical schemes to compute $\ol{H}$ have yet to be found. 
We refer to \cite{ATY1, ATY2, B1, Con1, Con2, CIPP, E, EvG, Fa, Gao, QTY} and the references therein for recent progress.

In this paper,  we aim to investigate the relation between $V$ and $\ol{H}$ from the perspective of the following inverse problem first formulated in \cite{LTY}.

\begin{quest}\label{q1}
Let $H\in C(\R^n)$ be a given coercive function, that is, $\lim_{|p| \to \infty} H(p) = +\infty$.
Let $V_1, V_2 \in C(\R^n)$ be two given potential energy functions which are $\Z^n$-periodic. 
Let $\ol{H}_1, \ol{H}_2$ be the effective Hamiltonians corresponding to the Hamiltonians $H(p)+V_1(x)$, $H(p)+V_2(x)$, respectively.
If 
\[
\ol{H}_1 \equiv \ol{H}_2,
\]
then what can we conclude about the relations between $V_1$ and $V_2$?

\end{quest}

When $n=1$,  a complete answer was provided in \cite{LTY} for a general class of convex $H$. It was shown that
\[
\text{$\ol{H}_1 \equiv \ol{H}_2$ \quad  $\iff$  \quad $V_1$ and $V_2$ have same distributions,}
\]
that is, $\displaystyle \int_{0}^1 f(V_{1}(x))\,dx=\int_{0}^1 f(V_{2}(x))\,dx$ for all $f \in C(\R)$. 

\smallskip

In case $n\geq 2$,  the only known $\overline H$-invariant transformations are translation and scaling, i.e., for some $c\in \mathbb Q \cap (0, +\infty)$ and $x_0 \in \T^n$,
$$
\quad V_1(x)=V_2\left({x\over c}+x_0\right)  \quad \text{for all } x\in \T^n \quad \Longrightarrow \quad \ol{H}_1 \equiv \ol{H}_2.
$$
If $H$ is convex and even,  the rescaling factor  $c$ could also be negative.  However,  If $H$ is even but nonconvex,  $c$ has to be  positive due to some pathological phenomena associated with nonconvexity (loss of evenness \cite{QTY}).  It is natural to investigate the following converse question. 
Throughout this paper,  we  focus on the mechanical Hamiltonian case, that is, the case where $H(p)={1\over 2}|p|^2$  for $p\in \R^n$.
\begin{quest}\label{q2}
Assume that $n\geq 2$ and $H(p)=\frac{1}{2}|p|^2$ for $p\in \R^n$.
Let $V_1, V_2 \in C(\R^n)$ be two given potential energy functions which are $\Z^n$-periodic. 
Let $\ol{H}_1, \ol{H}_2$ be the effective Hamiltonians corresponding to the Hamiltonians $H(p)-V_1(x)$, $H(p)-V_2(x)$, respectively.
If 
\[
\ol{H}_1 \equiv \ol{H}_2,
\]
then can we conclude that 
$$
\quad V_1(x)=V_2\left({x\over c}+x_0\right) \quad \text{for all } x\in \T^n,
$$
for some $c\in \mathbb{Q} \setminus\{0\}$ and $x_0\in \T^n$?

\end{quest}

Some results related to this question were established in \cite{LTY}. 
For example, if $V_1$ is constant, then the conclusion of Question \ref{q2} holds, that is, $V_2$ must  be the same constant (see \cite[Theorem 1.1]{LTY}).
In the general setting where $V_1, V_2 \in C^\infty(\T^n)$, by \cite[Theorem 1.2]{LTY}, $\ol{H}_1=\ol{H}_2$ implies that
\[
\int_{\Bbb T^n}V_1\,dx=\int_{\Bbb T^n}V_2\,dx,
\]
and under an extra decay condition of the Fourier coefficients of $V_1, V_2$, we also have  
\[
\int_{\Bbb T^n}V_{1}^{2}\,dx=\int_{\Bbb T^n}V_{2}^{2}\,dx.
\]
It was conjectured in \cite[Remark 1.1]{LTY} that under the settings of Question \ref{q2} and some further reasonable assumptions on $V_1, V_2$, 
if $\ol{H}_1=\ol{H}_2$,
then $V_1$ and $V_2$ have the same distribution. 
Clearly, this conjecture is weaker than the conclusion of Question \ref{q2}.
We address more about this point at the end of Subsection \ref{subsec:result}.

\medskip

It is natural to study the above questions in the case that $V_1$ and $V_2$ are trigonometric polynomials with $m$ mutually non-parallel Fourier modes.  
In this paper, as a preliminary step, we settle  Question \ref{q2}  when number of modes $m=3$.  
When $m\leq 2$, the analysis is much simpler.  
We give the main results in the following subsection.

\subsection{Main results} \label{subsec:result}

For $l=1,2$, set
\[
{\rm (A)} \quad
\begin{cases}
V_l(x)= a_{l0}+ \sum_{j=1}^m( \lam_{lj} e^{i 2 \pi k_{lj} \cdot x} + \overline{\lam_{lj}}e^{-i 2\pi k_{lj} \cdot x}),\\
\text{where $a_{l0} \in \R$, $\{\lam_{lj}\}_{j=1}^m \subset \mathbb{C}$ and $\{k_{lj}\}_{j=1}^m \subset \Z^n \setminus \{0\}$
such that}\\
\text{each pair of the $m$ vectors $\{k_{lj}\}_{j=1}^m$ are not parallel.}
\end{cases}
\]
Here, $ \overline{\lam_{1j}}$ is the complex conjugate of $\lam_{1j}$ for $1\leq j\leq m$. 
The following are our main results. 

\begin{thm}\label{thm:2d}
Assume that $m=3$, $n=2$, $H(p)=\frac{1}{2}|p|^2$ for all $p\in \R^2$, and {\rm (A)} holds.  Assume that 
\[
\ol{H}_1(p) = \ol{H}_2 (p) \quad \text{ for all $p\in  \R^2$}.
\]
Then there exist $c\in  \mathbb{Q} \setminus \{0\}$ and $x_0\in \mathbb{T}^2$ such that
\[
V_1(x)=V_2\left({x\over c}+x_0\right) \quad \text{ for all } x\in \T^2.
\]
\end{thm}

\begin{thm}\label{thm:3d} Assume that $m=3$, $n\geq 3$, $H(p)=\frac{1}{2}|p|^2$ for all $p\in \R^n$, and {\rm (A)} holds. 
There are three cases as following.

{\rm (1)}  If $\{k_{1j}\}_{j=1}^3$ are mutually orthogonal, then $\ol{H}_1= \ol{H}_2$ if and only if for $1\leq j\leq 3$,
\[
k_{1j}\parallel k_{2j}   \quad \mathrm{and} \quad |\lam_{1j}|=|\lambda_{2j}|.
\]

{\rm (2)} If $k_{11}\perp k_{12}$ and $k_{11}\perp k_{13}$, but $k_{12} \not \perp k_{13}$, then  $\ol{H}_1= \ol{H}_2$ if and only if
\[
k_{11}\parallel k_{21}, \quad c k_{12}= k_{22},  \ c  k_{13}=k_{23}  \quad \text{for some $c\in \mathbb{Q} \setminus \{0\}$,}
\]
and  for $1\leq j\leq 3$,
\[
 |\lam_{1j}|=|\lambda_{2j}|.
\]

{\rm (3)} If $\{k_{1j}\}_{j=1}^3$ do not satisfy {\rm (1)} and {\rm (2)} after permutations, then 
\[
\ol H_1\equiv \ol H_2  \quad \iff \quad V_1(x)=V_2\left({x\over c}+x_0\right) \quad \text{ for all } x \in \T^n,
\]
for some $c\in  \mathbb{Q} \setminus \{0\}$ and $x_0\in \mathbb{T}^n$.

\end{thm}

\begin{rem}   
Theorem \ref{thm:2d} can actually be viewed as a special case of (3) in Theorem \ref{thm:3d}. 
Nevertheless, the major part of this paper is devoted to  proving  this two dimensional result, and hence, it is worth stating it as a separate theorem.

\end{rem}

For completeness,  we also present the case when $m\leq 2$. 

\begin{thm}\label{thm:3d-2}  Assume that $m\leq 2$,  $H(p)=\frac{1}{2}|p|^2$ for all $p\in \R^n$, and {\rm (A)} holds. Then

{\rm (1)} If $m=1$, then
$$
\ol H_1\equiv \ol H_2  \quad \iff \quad V_1(x)=V_2\left({x\over c}+x_0\right) \quad \text{for all } x\in \T^n,
$$
for some $c\in  \mathbb{Q} \setminus \{0\}$ and $x_0\in \mathbb{T}^n$.

{\rm (2)} If $m=2$, then there are two cases.
\begin{itemize}
\item[(i)] If $k_{11}\perp k_{12}$, then then $\ol H_1\equiv \ol H_2$ if and only if  for $j=1,2$,
$$
k_{1j}\parallel k_{2j}   \quad \mathrm{and} \quad |\lam_{1j}|=|\lambda_{2j}|.
$$

\item[(ii)] If $k_{11}$ is not perpendicular to $k_{12}$, then 
$$
\ol H_1\equiv \ol H_2  \quad \iff \quad V_1(x)=V_2\left({x\over c}+x_0\right) \quad \text{for all } x\in \T^n,
$$
for some $c\in  \mathbb{Q} \setminus \{0\}$ and $x_0\in \mathbb{T}^n$.
\end{itemize}
\end{thm}

Theorems \ref{thm:2d}--\ref{thm:3d-2} settle the conjecture stated in \cite[Remark 1.1]{LTY} completely in case $m\leq 3$.
Of course, the case $m>3$ is still open.

We believe that the rigidity property should hold for ``generic"  periodic potentials in any dimension.  More precisely, 
we formulate the following conjecture.
\begin{conj}
We conjecture that

{\rm (1)} Theorem \ref{thm:2d} holds when $m\geq 3$, $n=2$.   

{\rm (2)} If $n\geq 3$,  then the result of Theorem \ref{thm:2d} is valid provided that $V_1, V_2$ belong to a dense open set of smooth periodic functions. 
\end{conj}

It is clear that this conjecture is stronger than that in  \cite[Remark 1.1]{LTY}, but under the caveat that we require a generic assumption on $V_1, V_2$.
Otherwise, it does not hold true (see parts (1)--(2) of Theorem \ref{thm:3d} and part (2)(i) of Theorem \ref{thm:3d-2} above).

\subsection {Outline of the paper.}  
In Section \ref{sec:prelim},  we give a quick review of the method of asymptotic expansions of $\overline{H}_1, \overline{H}_2$ at infinity introduced in \cite{LTY}
(see also \cite{JTY}). 
This is our main tool in studying the inverse problem.  
The proofs of our results will be given in Sections \ref{sec:2d} and \ref{sec:3d}.  
They involve delicate analysis combining plane geometry, linear algebra and trigonometric functions.


\section{Preliminary: Asymptotic expansion of $\ol{H}_1, \ol{H}_2$ at infinity} \label{sec:prelim}

 \subsection{Settings}
For $x \in \R^n$, we write $x=(x_1,x_2,...,x_n)$.

Assume there exists $m\in \N$  such that (A) holds.
Let us only perform calculations with respect to $\ol{H}_1$. 
In light of (A), $V_1$ satisfies that
\[
\begin{cases}
V_1(x)= a_{10}+\sum_{j=1}^m( \lam_{1j} e^{i 2\pi k_{1j} \cdot x} + \overline{\lam_{1j}}e^{-i 2 \pi k_{1j} \cdot x}),\\
\text{where $a_{10} \in \R$, $\{\lam_{1j}\}_{j=1}^m \subset \mathbb{C}$ and $\{k_{1j}\}_{j=1}^m \subset \Z^n \setminus \{0\}$
such that}\\
\text{each pair of the $m$ vectors $\{k_{1j}\}_{j=1}^m$ are not parallel.}
\end{cases}
\]
Here, $ \overline{\lam_{1j}}$ is the complex conjugate of $\lam_{1j}$ for $1\leq j \leq m$.



\subsection{Asymptotic expansion at infinity}

For a given vector $Q \neq 0$ and $\ep>0$, set $p =\frac{Q}{\sqrt{\ep}}$.
The cell problem for this vector $p$ is
\[
\frac{1}{2}\left|\frac{Q}{\sqrt{\ep}}+Dv^\ep_1\right|^2 + V_1(x) = \ol{H}_1\left(\frac{Q}{\sqrt{\ep}}\right) \quad \text{ in } \T^n.
\]
Here, $v^\ep_1 \in C(\T^n)$ is a solution to the above.
Multiply both sides by $\ep$ to yield
\begin{equation}\label{ae}
\frac{1}{2}|Q+ \sqrt{\ep} Dv^\ep_1|^2 + \ep V_1(x) =  \ep \ol{H}_1\left(\frac{Q}{\sqrt{\ep}}\right)=:\ol{H}^\ep(Q) \quad \text{ in } \T^n.
\end{equation}
Let us first use a formal asymptotic expansion to do computations.
We use the following ansatz
\[
\begin{cases}
\sqrt{\ep} v^\ep_1(x) = \ep v_{11}(x) + \ep^2 v_{12}(x) + \ep^3 v_{13}(x)+\cdots,\\
\ol{H}^\ep(Q) = a_0 +\ep a_1 + \ep^2 a_2 + \ep^3 a_3+\cdots.
\end{cases}
\]
Plug these into \eqref{ae} to imply
\[
\frac{1}{2} |Q+ \ep Dv_{11} + \ep^2 Dv_{12} +\cdots|^2  + \ep V_1  = \ol{H}^\ep(Q) = a_0 + \ep a_1 + \ep^2 a_2 +\cdots \quad \text{ in } \T^n.
\]
We first compare the $O(1)$ terms in both sides of the above equality to get 
\[
a_0= \frac{1}{2}|Q|^2.
\]
By using $O(\ep)$, we get 
\[
Q \cdot Dv_{11} + V_1 = a_1  \quad \text{ in } \T^n.
\]
Hence, $a_1=\int_{\T^n} V_1\,dx = a_{10}$ and
\begin{equation}\label{v11}
Dv_{11} = -\sum_{j=1}^m (\lam_{1j} e^{i 2 \pi k_{1j} \cdot x} + \ol{\lam_{1j}} e^{-i 2 \pi k_{1j} \cdot x}) \frac{k_{1j}}{k_{1j} \cdot Q}.
\end{equation}
Next, using $O(\ep^2)$, we achieve that
\begin{equation}\label{a2}
a_2 = \sum_{j=1}^m \frac{|\lam_{1j}|^2 |k_{1j}|^2}{ |k_{1j}\cdot Q|^2},
\end{equation}
and furthermore,
\begin{align*}
Q\cdot Dv_{12} &= a_2 - \frac{1}{2} |Dv_{11}|^2\\
&=-\frac{1}{2} \sum_{ \pm k_{1j} \pm k_{1l} \neq 0} \frac{\lam_{1j}^{\pm} \lam_{1l}^{\pm} k_{1j} \cdot k_{1l}}{(k_{1j}\cdot Q) (k_{1l}\cdot Q)} e^{i 2\pi (\pm k_{1j} \pm k_{1l})\cdot x}.
\end{align*}
Here for convenience, for $1\leq j \leq m$, we denote by
$$
\lam_{1j}^{+}=\lam_{1j} \quad \text{and} \quad \lam_{1j}^{-}=\overline{\lam_{1j}}.
$$  Thus,
\begin{equation*}\label{v12}
Dv_{12} = -\frac{1}{2} \sum_{ \pm k_{1j} \pm k_{1l} \neq 0} \frac{\lam_{1j}^{\pm}\lam_{1l}^{\pm} k_{1j} \cdot k_{1l}}{(k_{1j}\cdot Q) (k_{1l}\cdot Q)} e^{i 2\pi (\pm k_{1j} \pm k_{1l})\cdot x}
\frac{\pm k_{1j} \pm k_{1l}}{(\pm k_{1j} \pm k_{1l})\cdot Q}.
\end{equation*}
Let us now switch to a symbolic way of writing to keep track with all terms. 
Denote by $\sum_G$ to be a good sum where all terms are well-defined, that is, all denominators of the fractions in the sum are not zero. We have
\begin{equation}\label{v12}
Dv_{12} = -\frac{1}{2} \sum_{G} \frac{\lam_{1j_1}^{\pm} \lam_{1j_2}^{\pm} k_{1j_1} \cdot k_{1j_2}}{(k_{1j_1}\cdot Q) (k_{1j_2}\cdot Q)} e^{i 2\pi (\pm k_{1j_1} \pm k_{1j_2})\cdot x}
\frac{\pm k_{1j_1} \pm k_{1j_2}}{(\pm k_{1j_1} \pm k_{1j_2})\cdot Q}.
\end{equation}
Let us now look at $O(\ep^3)$:
\[
Q \cdot Dv_{13} = a_3 - Dv_{11}\cdot Dv_{12}.
\]
Hence,
\[
a_3 = \int_{\T^n} Dv_{11}\cdot Dv_{12}\,dx,
\]
and
\begin{multline}\label{v13}
Dv_{13} =- \frac{1}{2} \sum_{G} \frac{\lam_{1j_1}^{\pm} \lam_{1j_2}^{\pm} \lam_{1 j_3}^{\pm} (k_{1j_1} \cdot k_{1j_2})(\pm k_{1j_1} \pm k_{1j_2})\cdot k_{1 j_3}}{(k_{1j_1}\cdot Q) (k_{1j_2}\cdot Q)(k_{1j_3}\cdot Q)(\pm k_{1j_1} \pm k_{1j_2})\cdot Q}  \times\\
\times e^{i 2\pi (\pm k_{1j_1} \pm k_{1j_2} \pm k_{1j_3})\cdot x}
\frac{\pm k_{1j_1} \pm k_{1j_2} \pm k_{1 j_3}}{(\pm k_{1j_1} \pm k_{1j_2} \pm k_{1 j_3})\cdot Q}.
\end{multline}
The $O(\ep^4)$ term yields
\[
Dv_{11} \cdot Dv_{13} + \frac{1}{2} |Dv_{12}|^2 + Q\cdot Dv_{14} = a_4.
\]
Integrate to get
$$
a_4= \frac{1}{2}  \int_{\T^2}|Dv_{12}|^2 \,dx+ \int_{\T^2} Dv_{11} \cdot Dv_{13} \,dx.
$$
The first integral in the formula of $a_4$  contains terms like $I(j_1,j_2)+I(j_3,j_4)+II(j_i,j_2,j_3,j_4)$ with 
$$
I(j_1,j_2)=\frac{1}{8}\frac{|\lam_{1j_1}|^2 |\lam_{1j_2}|^2 |k_{1j_1} \cdot k_{1j_2}|^2 |\pm k_{1j_1} \pm k_{1j_2}|^2}{|k_{1j_1}\cdot Q|^2 |k_{1j_2}\cdot Q|^2 |(\pm k_{1j_1} \pm k_{1j_2})\cdot Q|^2},
$$
and $I(j_3,j_4)$ is of the exact same form with $(j_3,j_4)$ in place of $(j_1,j_2)$. Besides,
$$
II(j_i,j_2,j_3,j_4)= \frac{1}{8} {\lam_{1j_1}^{\pm} \lam_{1j_2}^{\pm} \lam_{1j_3}^{\pm} \lam_{1j_4}^{\pm} (k_{1j_1} \cdot k_{1j_2})(k_{1j_3} \cdot k_{1j_4})\over (k_{1j_1}\cdot Q) (k_{1j_2}\cdot Q)(k_{1j_3}\cdot Q) (k_{1j_4}\cdot Q)}\cdot { |\pm k_{1j_1} \pm k_{1j_2}|^2\over |(\pm k_{1j_1} \pm k_{1j_2})\cdot Q|^2} 
$$
provided that $(j_1,j_2) \neq (j_3,j_4)$ and $\pm k_{1j_1} \pm k_{1j_2} \pm k_{1 j_3} \pm k_{1 j_4} =0$. 

It is more important noticing that the terms that are not vanished in the above second integral of $a_4$ are the ones that have $\pm k_{1j_1} \pm k_{1j_2} \pm k_{1 j_3} \pm k_{1 j_4} =0$.
Hence, $\pm k_{1j_1} \pm k_{1j_2} \pm k_{1 j_3} = \mp k_{1 j_4}$ and these terms look like
\begin{equation}\label{v11-v13}
 \frac{\lam_{1j_1}^{\pm} \lam_{1j_2}^{\pm} \lam_{1 j_3} ^{\pm}\lam_{1 j_4}^{\pm} (k_{1j_1} \cdot k_{1j_2})[(\pm k_{1j_1} \pm k_{1j_2})\cdot k_{1 j_3}] |k_{1j_4}|^2}{(k_{1j_1}\cdot Q) (k_{1j_2}\cdot Q)(k_{1j_3}\cdot Q)[(\pm k_{1j_1} \pm k_{1j_2})\cdot Q] |k_{1j_4}\cdot Q|^2}.
\end{equation}
Of course, $v_{14}$ satisfies
\begin{equation}\label{v14}
Q\cdot Dv_{14} = a_4 -Dv_{11} \cdot Dv_{13} - \frac{1}{2} |Dv_{12}|^2 .
\end{equation}

By computing in an iterative way, we can get formulas of $a_l$ and $v_{1l}$ for all $l\in \N$.
It turns out that this formal asymptotic expansion of $\ol{H}^\ep(Q)$ holds true rigorously.
For our purpose here, we only need the first five terms in the expansion.

\begin{prop}\label{prop:a4}
Assume that $H(p)=\frac{1}{2}|p|^2$ for all $p\in \R^n$ and {\rm (A)} holds.
Let  $\ol{H}_1$ be the effective Hamiltonian corresponding to the Hamiltonian $H(p)+V_1(x)$.
Let $Q \neq 0$ be a vector in $\R^n$ such that $Q$ is not perpendicular to 
each nonzero vector of  $k_{1j_1}, \pm k_{1j_1} \pm k_{1j_2}, \pm k_{1j_1} \pm k_{1j_2} \pm k_{1 j_3}$ 
and $ \pm k_{1j_1} \pm k_{1j_2} \pm k_{1 j_3} \pm k_{1 j_4}$ for 
$1\leq j_1, j_2, j_3, j_4 \leq m$.

For $\ep>0$, set $\ol{H}^\ep(Q) = \ep \ol{H}_1\left(\frac{Q}{\sqrt{\ep}}\right)$. 
Then we have that, as $\ep \to 0$,
\[
\ol{H}^\ep(Q) = \frac{1}{2}|Q|^2 +\ep a_1+ \ep^2 a_2 + \ep^3 a_3 + \ep^4 a_4 + O(\ep^5).
\]
Here the error term satisfies $|O(\ep^5)| \leq K \ep^5$ for some $K$ depending only on $Q$, $\{\lam_{1j}\}_{j=1}^m$
and $\{k_{1j}\}_{j=1}^m$.
\end{prop}
Let us present the proof of this proposition here for the sake of completeness.
A version of this was presented in \cite[Proof of Theorem 1.2 (Part 3)]{LTY}.
See also \cite[Lemma 3.1]{JTY}.
\begin{proof}
Let $v_{11}, v_{12}, v_{13}, v_{14}$ be solutions to \eqref{v11}, \eqref{v12}, \eqref{v13}, \eqref{v14}, respectively.
Let $\phi= \ep v_{11}+\ep^2 v_{12}+ \ep^3 v_{13}+\ep^4 v_{14}$, then $\phi$ satisfies
\[
\frac{1}{2}|Q+D\phi|^2 + \ep V_1 =  \frac{1}{2}|Q|^2 +\ep a_1+ \ep^2 a_2 + \ep^3 a_3 + \ep^4 a_4 + O(\ep^5) \quad \text{ in } \T^n.
\]
Recall that $w=\sqrt{\ep} v_1^\ep$ is a solution to \eqref{ae}.
By looking at the places where $w-\phi$ attains its maximum and minimum and using the definition of viscosity solutions, we arrive at the conclusion.
\end{proof}

We prepare some further definitions.
Denote 
\[
\begin{cases}
A_1=\left \{\pm k_{1j}, \ \pm k_{1j}\pm k_{1l} \,:\, 1\leq j, l \leq m  \text{ and $k_{1j}\cdot k_{1l}\not=0$}  \right\},\\
 A_2=\left\{\pm k_{2j},\ \pm k_{2j}\pm k_{2l} \,:\,  1\leq j, l \leq m  \text{ and $k_{2j}\cdot k_{2l}\not=0$} \right\}.
\end{cases}
\]
In other words, if $k_{1j} \cdot k_{2j} =0$ for some $i, j \in \{1,\ldots, m\}$,
then we do not collect $\pm k_{1j} \pm k_{2j}$ in $A_1$.
\begin{defn}[Sole vectors]
A vector $\alpha k_{1j_1} +\beta k_{1j_2}$, where $\alpha, \beta \in \{-1,1\}$ and $1 \leq j_1, j_2 \leq m$,
 is called  a sole vector  from $A_1$ if  it is in $A_1$ and is not equal to any other vectors in $A_1$.
 
 A vector $\alpha k_{2j_1} +\beta k_{2j_2}$, where $\alpha, \beta \in \{-1,1\}$ and $1 \leq j_1, j_2 \leq m$,
 is called  a sole vector  from $A_2$ if   it is in $A_2$ and is not equal to any other vectors in $A_2$.
 \end{defn}

\begin{rem}\label{singular} If $\alpha k_{1 j_1}+\beta k_{1 j_2}$  is a sole vector from $A_1$, then
$$
\frac{1}{4}\frac{|\lam_{1j_1}|^2 |\lam_{1j_2}|^2 |k_{1j_1} \cdot k_{1j_2}|^2 |\alpha k_{1j_1} +\beta k_{1j_2}|^2}{|k_{1j_1}\cdot Q|^2 |k_{1j_2}\cdot Q|^2 |(\alpha k_{1j_1} +\beta k_{1j_2})\cdot Q|^2}
$$
is the only term in $a_4$ containing $ {1\over |(\alpha k_{1j_1} +\beta  k_{1j_2})\cdot Q|^2}$.

\end{rem}

\begin{defn} Let $A_1$ and $A_2$ are two sets of vectors in $\mathbb{R}^n$. We write 
\[
A_1\prec A_2
\]
if for any $u\in A_1\setminus \{0\}$, there exists $v\in A_2 \setminus \{0\}$ such that $u\parallel v$.
\end{defn}

\begin{rem}\label{rem:sole}

If $\overline H_1\equiv \overline H_2$,  then Remark \ref{singular}, together with Proposition \ref{prop:a4}, implies  that 
$$
\begin{cases}
\left\{\text{Sole vectors from $A_1$} \right\}\prec A_2\\
\left\{\text{Sole vectors from $A_2$} \right\}\prec A_1.
\end{cases}
$$
Heuristically, this  could lead to an over-determined linear system, which plays a key role in proving our rigidity results. 
\end{rem}

\section{Proof of theorem \ref{thm:2d}} \label{sec:2d}

In this section, we always assume that the settings in Theorem \ref{thm:2d} are in force. 
In particular, we have $n=2$ and $m=3$. 
Without loss of generality,  we assume further that for $l=1,2$, 
\begin{itemize}
\item[(H)]  $ k_{l1}, \ k_{l2},\  k_{l3}$ are aligned in the counter-clockwise order on the upper half plane $\{x=(x_1,x_2)\,:\,x_2 \geq 0\}$.
\end{itemize}
See Figure \ref{fig:V1} below.

\begin{figure}[h]
\begin{center}
\begin{tikzpicture}

\draw[->] (-0.5,0)--(4,0);
\draw[->] (0,-2)--(0,4);
\draw (4,-0.2) node {$x_1$};
\draw (-0.2,4) node {$x_2$};

\draw[blue, ->] (0,0)--(3,1);
\draw[blue, dashed, ->] (0,0)--(-3,-1);
\draw (3.3,1) node {$k_{11}$};

\draw[red,->] (0,0)--(2,2);
\draw[red,dashed,->] (0,0)--(-2,-2);
\draw (2.3,2) node{$k_{12}$};

\draw[cyan,->] (0,0)--(-1,4);
\draw[cyan, dashed, ->] (0,0)--(1,-4);
\draw (-1.3,4) node{$k_{13}$};

\end{tikzpicture}
\caption{The vectors $\{k_{1j}\}_{j=1}^3$}  \label{fig:V1}
\end{center}
\end{figure}
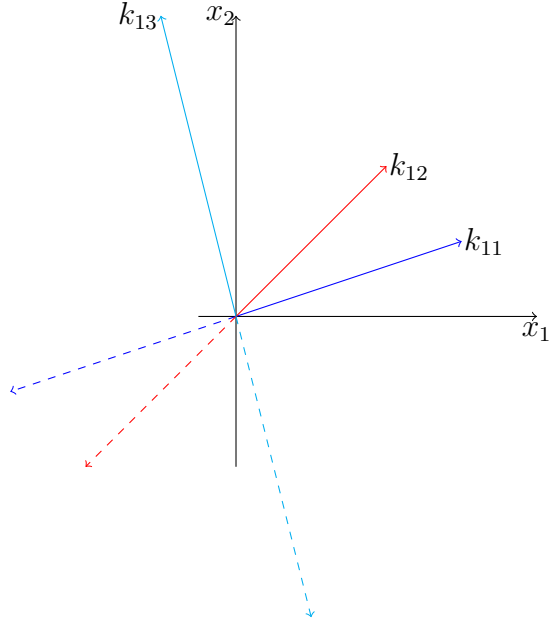

We proceed to prove Theorem \ref{thm:2d} via the following lemmas.

\begin{lem}\label{lem:2d-1}
Assume that the settings in Theorem \ref{thm:2d} hold. Then $a_{10}=a_{20}$ and, for all $1\leq j \leq 3$,
\[
 |\lam_{1j}| = |\lam_{2j}| \quad \text{and} \quad \frac{k_{1j}}{|k_{1j}|} = \frac{k_{2j}}{|k_{2j}|}.
\]
\end{lem}

\begin{proof}
We use the asymptotic expansion of $\ol{H}^\ep(Q)$ in Proposition \ref{prop:a4} and compare the coefficients to get the conclusion.
Firstly, by comparing $a_1$, we imply $a_{10}=a_{20}$ immediately.

 Secondly, we use the formula of $a_2$ given in \eqref{a2} to get
\[
 \sum_{j=1}^3 \frac{|\lam_{1j}|^2 |k_{1j}|^2}{ |k_{1j}\cdot Q|^2}= \sum_{j=1}^3 \frac{|\lam_{2j}|^2 |k_{2j}|^2}{ |k_{2j}\cdot Q|^2}.
\]
Fix $j\in \{1,2,3\}$. By letting $Q \to k_{1j}^\perp$, we use (H) to conclude
\begin{equation}\label{3mode-1}
\frac{k_{1j}}{|k_{1j}|} = \frac{k_{2j}}{|k_{2j}|} \quad \text{and} \quad |\lam_{1j}| = |\lam_{2j}|.
\end{equation}
\end{proof}

Thanks to Lemma \ref{lem:2d-1}, for $1 \leq j \leq 3$, there exists $\alpha_j>0$ such that
 \[
 k_{1j}=\alpha_j k_{2j}.
 \]
 The following is a result in linear algebra (or plane geometry), which we believe is of independent interest.

\begin{lem}\label{lem:2d-2}
For $j=1,2,3$, let $\alpha_j>0$ be a given number. 
Let $u_1$, $u_2$ and $u_3$ be non-parallel vectors on the upper half plane $\{x=(x_1,x_2)\,:\,x_2 \geq 0\}$, which are aligned in the counter-clockwise order. 
Set
\[
\begin{cases}
S_1=\{\pm u_i,\ \pm u_i\pm u_j \,:\,  1\leq i, j \leq 3 \text{ and $u_i\cdot u_j\not=0$}\},\\
 S_2=\{\pm u_i, \  \pm\alpha_i u_i \pm \alpha_j u_j \,:\,  1\leq i<j\leq 3 \text{ and $u_i\cdot u_j\not=0$}\}.
\end{cases}
\]

 If
\[
\begin{cases}
\{\text{Sole vectors from $S_1$} \}\prec S_2,\\
\{\text{Sole vectors from $S_2$} \}\prec S_1,
\end{cases}
\]
then $\alpha_1=\alpha_2=\alpha_3$.
\end{lem}

\begin{proof}

We may normalize $\alpha_1=1$.  For $1\leq i, j \leq 3$, denote 
$$
a_{ij}=u_i\times u_j={\rm det}[u_i,u_j].
$$  
Set $\vec{a}=(a_{12}, a_{13}, a_{23}) \in \R^3$. 

We prove by contradiction by assuming  that $\alpha_2$ and $\alpha_3$ are not both $1$. 
This is rather a lengthy proof and we divide it into steps in order to keep track with the key points easily.  
The directions $\{\pm  u_j\}_{j=1}^3$ divide $\R^2$ into six regions named I--VI as in Figure \ref{fig:case1} below.

\medskip

 \noindent {\bf   Part I: Non-orthogonal case.}  We first assume that $u_1, u_2, u_3$ are mutually non-orthogonal. Then it is easy to see that  
\[
\left\{\pm (u_1+u_2),\  \pm (u_2+u_3), \pm (u_3-u_1)\right\}\subseteq \{\text{Sole vectors from $S_1$} \}
\]
and
\[
\left\{\pm (u_1+\alpha _2 u_2),\  \pm (\alpha _2 u_2+\alpha_3 u_3), \ \pm (\alpha_3 u_3-  u_1)\right\}\subseteq \{\text{Sole vectors from $S_2$} \}.
\]

\medskip

{\bf Step 1.}  Assume that $\alpha_2=1$ but $\alpha_3\not= 1$. Then we have that 

\[
\begin{cases}
u_2+u_{3}\parallel u_1+\alpha_3 u_3 \ \mathrm{or} \  u_1-  u_2,\\
u_3-u_1\parallel  u_1- u_2 \ \mathrm{or} \ \alpha_3 u_3-  u_2
\end{cases}
\]
and
\[
\begin{cases}
u_2+\alpha_3 u_3\parallel u_1+u_3  \ \mathrm{or} \  u_1-u_2,\\
\alpha_3 u_3 - u_1\parallel u_1-u_2 \ \mathrm{or} \  u_3-u_2
\end{cases}
\]
Since $u_2+u_3$, $u_3-u_1$, $u_2+\alpha_3 u_3$ and $\alpha_3 u_3 - u_1$ are mutually non-parallel, there are only two possibilities. 

\medskip

{\it Case 1.1.} None of these four vectors is parallel to $u_1-u_2$. Then
$$
\begin{cases}
u_2+u_{3}\parallel  u_1+\alpha_3 u_3\\
u_3-u_1\parallel \alpha_3 u_3- u_2\\
 u_2+\alpha_3 u_3\parallel u_1+u_3\\
\alpha_3 u_3 -  u_1\parallel u_3-u_2.
\end{cases}
$$
We use the fact that $u\times \hat u=0$ provided $u\parallel \hat u$ to  yield
$$
\vec{a}\cdot w_{k}=0  \quad \text{for all  $k={1,2,3,4}$}. 
$$
Here
$$
w_1=(-1,-1,\alpha_3), \ w_2=(1,-\alpha_3,1),\  w_3=(-1,-\alpha_3, 1),  \ w_4=(1,-1, \alpha_3).
$$
Therefore, the dimension of  $V=\text{span}\{w_1, w_2, w_3, w_4\}$ is  at most $2$.  Therefore  $\mathrm{det}|w_1,w_2,w_4|=0$, which leads to  $\alpha_3=1$. This is a contradiction.

\medskip

{\it Case 1.2.} One and only one of these four vectors is parallel to $u_1-u_2$.  
As the roles of $u_3$ and $\alpha_3 u_3$ are the same,
we only need to consider two situations . 
Either

$$
\begin{cases}
u_2+u_{3}\parallel u_1-u_2\\
u_3-u_1\parallel \alpha_3 u_3-  u_2\\
u_2+\alpha_3 u_3\parallel u_1+u_3\\
\alpha_3 u_3 - u_1\parallel u_3-u_2.
\end{cases}
\quad \mathrm{or} \quad \
\begin{cases}
u_2+u_{3}\parallel  u_1+\alpha_3 u_3\\
u_3-u_1\parallel  u_1-u_2\\
 u_2+\alpha_3 u_3\parallel u_1+u_3\\
\alpha_3 u_3 - u_1\parallel u_3-u_2.
\end{cases}
$$
Then we have either the dimension of $\text{span}\{\hat w_1, w_2, w_3, w_4\}$ is $2$ 
or  the dimension of $\text{span}\{w_1, \hat w_2, w_3, w_4\}$ is $2$. 
Here $\hat w_1=(-1,-1, 1)$ and $\hat w_2=(1,-1, 1)$.  Both cases lead to the same conclusion that $\alpha_3=1$. This is  a contradiction.

\medskip

{\bf Step 2.}  Either  $\alpha_2\not=\alpha_3=1$  or $\alpha_2=\alpha_3\not=1$.
This case can be transformed back to the previous case by suitable rotations, reflections and normalizations. 

\medskip

{\bf Step 3:} Now we consider the case $1\not=\alpha_2\not=\alpha_3\not=1$. Then we must have that  for $i,j\in \{1,2,3\}$
\[
\begin{cases}
u_i+u_j  \nparallel \alpha_i u_i+ \alpha_j u_j,\\
u_i-u_j  \nparallel \alpha_i u_i-\alpha_j u_j.
\end{cases}
\]
Accordingly, 
\begin{equation}
\begin{cases}
u_1+u_2  \parallel   u_1+ \alpha_3 u_3\  \mathrm{or} \ \alpha_3 u_3-\alpha_2 u_2,\\
u_2+u_3 \parallel   u_1+ \alpha_3 u_3\  \mathrm{or}  \ \alpha_2 u_2 -  u_1,\\
u_3-u_1 \parallel  \ \alpha_2 u_2 -   u_1\  \mathrm{or}  \  \alpha_3 u_3- \alpha_2 u_2.
\end{cases}
\end{equation}

Since $u_1+u_2$, $u_2+u_3$ and $u_3-u_1$ are mutually linearly independent,  we have only two scenarios.

\begin{equation}\label{step1}
\begin{cases}
u_1+u_2  \parallel  u_1+ \alpha_3 u_3\\
u_2+u_3  \parallel  \alpha_2 u_2 -  u_1,\\
u_3-u_1  \parallel   \alpha_3 u_3- \alpha_2 u_2.
\end{cases}
\quad  \mathrm {or}  \quad 
\begin{cases}
u_1+u_2  \parallel  \alpha_3u_3-\alpha_2u_2\\  
u_2+u_3  \parallel   u_1+ \alpha_3u_3,\\
u_3-u_1  \parallel   \alpha_2u_2 - u_1.
\end{cases}
\end{equation}

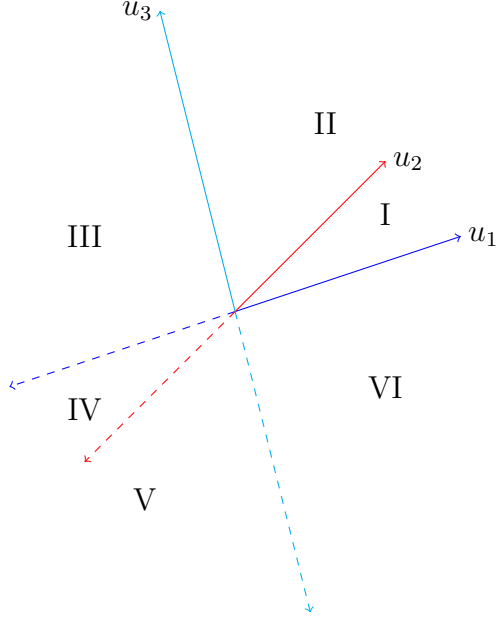
\begin{figure}[h]
\begin{center}
\begin{tikzpicture}

\draw[blue, ->] (0,0)--(3,1);
\draw[blue, dashed, ->] (0,0)--(-3,-1);
\draw (3.3,1) node {$u_1$};

\draw[red,->] (0,0)--(2,2);
\draw[red,dashed,->] (0,0)--(-2,-2);
\draw (2.3,2) node{$u_2$};

\draw[cyan,->] (0,0)--(-1,4);
\draw[cyan, dashed, ->] (0,0)--(1,-4);
\draw (-1.3,4) node{$u_3$};

\draw (2,1.3) node{I};
\draw (-2,-1.3) node{IV};
\draw (1.2,2.5) node{II};
\draw (-1.2,-2.5) node{V};
\draw (-2,1) node{III};
\draw (2,-1) node{VI};

\end{tikzpicture}
\caption{Vectors $\{u_j\}_{j=1}^3$ and six regions I--VI} \label{fig:case1}
\end{center}
\end{figure}

Similarly, there are two other cases to be considered for $u_1+ \alpha_2 u_2$, $\alpha_2 u_2 + \alpha_3 u_3$, $\alpha_3 u_3 - u_1$.

\begin{equation}\label{step2}
\begin{cases}
u_1+\alpha_2u_2  \parallel u_1+ u_3\\
\alpha_2u_2+\alpha_3u_3  \parallel  u_2 - u_1,\\
\alpha_3u_3-u_1  \parallel   u_3- u_2.
\end{cases}
\quad  \mathrm {or}  \quad 
\begin{cases}
u_1+\alpha_2u_2  \parallel  u_3-u_2\\  
\alpha_2u_2+\alpha_3u_3  \parallel  u_1+ u_3,\\
\alpha_3u_3- u_1  \parallel   u_2 - u_1.
\end{cases}
\end{equation}
In total, there are four cases to be studied.

\medskip
{\it Case 3.1.} Assume that
\begin{equation}\label{step2}
\begin{cases}
u_1+u_2  \parallel u_1+ \alpha_3u_3\\
u_2+u_3  \parallel  \alpha_2u_2 - u_1,\\
u_3-u_1  \parallel   \alpha_3u_3- \alpha_2u_2.
\end{cases}
\quad  \mathrm {and}  \quad 
\begin{cases}
u_1+\alpha_2u_2  \parallel u_1+ u_3\\
\alpha_2u_2+\alpha_3u_3  \parallel  u_2 - u_1,\\
\alpha_3u_3-u_1  \parallel   u_3- u_2.
\end{cases}
\end{equation}
 Considering cross product between parallel vectors,  we get that 
$$
\vec{a} \cdot v_i =0.
$$
for (here we write $\alpha=\alpha_2$ and $\beta=\alpha_3$)
$$
v_1=(-1,\beta, \beta), \ v_2=(1,1,-\alpha), \ v_3=(\alpha, -\beta, \alpha)
$$
and
$$
 v_4=(-\alpha, 1, \alpha), \  v_5=(\alpha, \beta, -\beta),  \ v_6=(1,-1,\beta).
$$
Clearly, the dimension of $\mathrm{span}\{v_1,v_2,v_3, v_4, v_5,v_6\}$ is 2. 
By noting that $v_2+v_3=(1+\alpha, 1-\beta, 0)$ and $v_5+v_6=(\alpha+1, \beta-1, 0)$,
we imply $v_2+v_3$ and $v_5+v_6$ are linearly dependent. 
Otherwise,  $\mathrm{span}\{v_1,v_2,v_3, v_4, v_5,v_6\}\subseteq \{x_3=0\}$,  which is impossible. 
Hence we obtain that $1-\beta =\beta-1$, that is, $\beta=1$. This is a contradiction.

\medskip

{\it Case 3.2.} We have that
\[
\begin{cases}
u_1+u_2  \parallel  u_1+ \alpha_3u_3\\
u_2+u_3  \parallel  \alpha_2u_2 -  u_1,\\
u_3-u_1  \parallel   \alpha_3u_3- \alpha_2u_2.
\end{cases}
\quad  \mathrm {and}  \quad \
\begin{cases}
 u_1+\alpha_2u_2  \parallel u_2- u_3,\\
\alpha_2u_2+\alpha_3u_3  \parallel u_1+ u_3,\\
\alpha_3u_3- u_1  \parallel u_2 - u_1.
\end{cases}
\]
Set
$$
\tilde v_4=(1,-1,-\alpha), \  \tilde v_5=(-\alpha, -\beta, \alpha),\  \tilde v_6=(-1, \beta, -\beta).
$$
Similarly, the rank of $\{v_1,v_2,v_3, \tilde v_4, \tilde v_5,\tilde v_6\}$ is 2.  Note that $v_2+v_3=(1+\alpha, 1-\beta, 0)$ and  $v_2+\tilde v_5=(1-\alpha, 1-\beta, 0)$. 
By the same argument as above, 
$v_2+v_3$ and $v_2+\tilde v_5$ are linearly dependent, which leads to $\beta=1$.
We again arrive at a contradiction.

\medskip

{\it Case 3.3.} We have that
\[\begin{cases}
u_1+u_2  \parallel  \alpha_3u_3-\alpha_2u_2\\  
u_2+u_3  \parallel  u_1+ \alpha_3u_3,\\
u_3-u_1  \parallel   \alpha_2u_2 - u_1.
\end{cases}
\quad  \mathrm {and}  \quad 
\begin{cases}
u_1+\alpha_2u_2  \parallel u_2- u_3,\\
\alpha_2u_2+\alpha_3u_3  \parallel u_1+ u_3,\\
\alpha_3u_3- u_1  \parallel u_2 - u_1.
\end{cases}
\]
Set
$$
\hat v_1=(-\alpha, \beta, \beta),\  \hat v_2=(-1, -1, \beta),  \ \hat v_3=(-\alpha, 1, -\alpha).
$$
Again,  the rank of $\{\hat v_1,\hat v_2,\hat v_3, \tilde v_4, \tilde v_5,\tilde v_6\}$ is 2.  
Note that $\tilde v_4+\tilde v_5=(1-\alpha, -1-\beta, 0)$ and  $\hat v_1-\hat v_2=(1-\alpha, \beta+1, 0)$. 
Similar to the above, $\tilde v_4+\tilde v_5$ and $\hat v_1-\hat v_2$ must be linearly dependent, which leads to  $\alpha=1$.  This is again a contradiction.

\medskip

Due to the symmetry, the remaining case is essentially the same as  Case 3.2. 
We omit the proof.

\medskip

\noindent {\bf Part II: Orthogonal Case.} Without loss of generality,  we assume the $u_1\perp u_3$. 
The other two situations ($u_1\perp u_2$ or $u_2\perp u_3$) can be converted into this case by suitable reflections and rotations.  For this case, 
\[
\left\{\pm (u_1+u_2),\  \pm (u_2+u_3)\right\}
\]
and
\[
\left\{\pm ( u_1+\alpha _2u_2),\  \pm (\alpha _2u_2+\alpha_3u_3)\right\}
\]
are still sole vectors of $S_1$ and $S_2$, respectively.  Also, it is important to note that, by definitions,
$$
\pm u_1\pm u_3\notin S_1   \quad \mathrm{and} \quad \pm u_1\pm \alpha_3 u_3\notin S_2.
$$
We consider two cases.
\medskip

{\it Case II.1.}  Assume that $1=\alpha_2\not=\alpha_3$. Then $\alpha_1 u_1+\alpha_2 u_2=u_1+u_2$. By the assumption
\begin{equation}\label{oneway}
\begin{cases}
u_2+u_3\parallel  u_1-u_2\\
u_2+\alpha_3 u_3\parallel u_1-u_2.
\end{cases}
\end{equation}
This leads to $u_2+ u_3\parallel u_2+\alpha_3 u_3$, which is absurd. 

\medskip

{\it Case II. 2.} Assume that  $\alpha_2\not=1$.  By the assumption,  we must that 
 $$
\begin{cases}
u_1+u_2\parallel  \alpha_3 u_3-\alpha _2 u_2\\
\alpha_3 u_3+\alpha _2 u_2\parallel u_2-u_1.
\end{cases}
$$
This is equivalent to 
$$
\begin{cases}
u_1+u_2\parallel  u_3-{\alpha_2\over \alpha_3} u_2\\
u_2-u_1\parallel u_3+{\alpha_2\over \alpha_3} u_2.
\end{cases}
$$
Then $-ra_{12}+a_{13}+a_{23}=-ra_{12}-a_{13}+a_{23}=0$ for $r={\alpha_2\over \alpha_3}$. This implies that $a_{13}=0$, i.e., $u_1\parallel u_3$, which is again absurd. The proof is complete.
\end{proof}

Combining Remark \ref{rem:sole} and the above Lemma \ref{lem:2d-2},  we obtain that there exists $c\in \mathbb{Q}$ such that for $j=1,2,3$,
\begin{equation}\label{k1j-k2j}
k_{2j}=ck_{1j}.
\end{equation}
Without loss of generality,  we set $c=1$. 
Note however that Lemma \ref{lem:2d-1} only gives us that $|\lam_{1j}| = |\lam_{2j}|$ for $1\leq j \leq 3$,
which is not yet enough to conclude Theorem \ref{thm:2d}.
To finish the proof,  we need one more relation between $\{\lambda_{1j}\}_{j=1}^3$ and $\{\lambda_{2j}\}_{j=1}^3$. 

 Since $\ol{H}_1=\ol{H}_2$, we get that
\begin{equation}\label{min-V}
\max_{\T^2} V_1=\overline H_1(0)=\overline H_2(0)=\max_{\T^2} V_2.
\end{equation} 
We use this relation to get the final piece of information.
Before doing so, we need some preparations.
\medskip
\begin{defn}
Given $r_1, r_2,r_3>0$ and $\alpha_1,  \alpha_2\in \mathbb{Q}$,  denote
$$
M(t)=\max_{\theta_1,\theta_2\in \mathbb{R}}\{r_1\cos \theta_1+r_2\cos \theta_2+r_3\cos (\alpha_1 \theta_1+\alpha_2 \theta_2+t)\} \quad \text{ for } t \in \R.
$$
Of course $M(t)$ depends on the parameters $r_1, r_2, r_3, \alpha_1, \alpha_2$, but we do not write down this dependence explicitly unless there is some confusion.
\end{defn}

It is easy to see that $\max_{\mathbb{R}}M=r_1+r_2+r_3$, and the maximum is attained when 
$$
t=2m\pi+2m_1\alpha_1 \pi+2m_2 \alpha_2 \pi
$$
for  $m, m_1, m_2\in \mathbb{Z}$.   
Note that the function $x \mapsto \cos x$ does not have non-global local maximum. 
We now show that this fact is also true for $M(t)$.  

\begin{lem}  \label{lem:M-min}
Every local maximum of $M$ is a global maximum.  
\end{lem}

\begin{proof} 
Suppose that $t_0$ is a local maximum of $M$. Assume that 
$$
M(t_0)=r_1\cos \theta_{1,0}+r_2\cos \theta_{2,0}+r_3 \cos (\alpha_1 \theta_{1,0}+\alpha_2 \theta_{2,0}+t_0).
$$
for some $\theta_{1,0}, \theta_{2,0} \in  \mathbb{R}$.  Then we must have that 
$$
\cos \theta_{1,0}=\cos \theta_{2,0}=\cos (\alpha_1 \theta_{1,0}+\alpha_2 \theta_{2,0}+t_0)=1.
$$
Otherwise, we can easily perturb $\theta_{1,0}, \theta_{2,0}$ and $t_0$ a bit to get a greater value of $M$ near $t_0$.  
\qedhere
\end{proof}

Now set
$$
l=\min\left\{|m\pi+m_1\alpha_1 \pi+m_2 \alpha_2 \pi| \,:\,  |m\pi+m_1\alpha_1 \pi+m_2 \alpha_2 \pi|>0,\  m, m_1, m_2\in \mathbb{Z}\right\}.
$$
Clearly, $l>0$ and, for all $t\in \R$,   
\begin{equation}\label{M-periodic}
M(t)=M(2l+t)=M(-t)=M(2l-t).
\end{equation}

\begin{prop} \label{prop:M}
The function $M$ is strictly decreasing on $[0, l]$, and is strictly increasing on $[l,2l]$.
\end{prop}

\begin{proof}
Thanks to Lemma \ref{lem:M-min} and the choice of $l$,  $M$ has no local maximum in $(0, 2l)$.  

Besides, \eqref{M-periodic} gives that $M(t) = M(2l-t)$ for all $t \in (0,2l)$, and thus,   $M$ cannot have any  local minimum in $(0, l)$.  
The proof is complete.
\end{proof}

The following is an immediate implication from Proposition \ref{prop:M} and \eqref{M-periodic}.
\begin{cor}\label{key-fact}
For $t_1, t_2 \in \R$,
$$
M(t_1)=M(t_2)
$$
 if and only if  $t_1=t_2+2kl$ or $t_1=2kl-t_2$ for some $k\in  \mathbb{Z}$.
\end{cor}

We are now ready to prove the main result.

\begin{proof}[{\bf Proof of Theorem \ref{thm:2d}}]  
Thanks to \eqref{k1j-k2j} and the normalization that $c=1$, we have $k_{1j} = k_{2j}$ for all $1\leq j \leq 3$.
We now write $k_j=k_{1j}=k_{2j}$ for simplicity for all $1\leq j \leq 3$.
Then
$$
V_1(x)=a_1+ \sum_{j=1}^{3}\left(\lambda_{1j} e^{ i 2\pi k_j\cdot x}+\ol{\lambda_{1j}} e^{- i 2\pi  k_j\cdot x}\right).
$$
Since $k_1, k_2, k_3$ are mutually non-parallel,  by translation (i.e., $x\mapsto x+x_0$ for suitable $x_0$), we may assume that 
$$
V_2(x)=a_1+\sum_{j=1}^{2}\left(\lambda_{1j} e^{ i 2\pi k_j\cdot x}+\ol{\lambda_{1j}}  e^{- i 2\pi k_j\cdot x}\right)
+\tilde \lambda_{23} e^{ i 2\pi k_3\cdot x}+\ol{\tilde  \lambda_{23}}  e^{- i 2\pi k_3\cdot x}.
$$
Denote $\lambda_{1j}=r_j e^{ i \omega_j}$  for $1\leq j\leq 3$ and $\tilde \lambda_{23}=r_3e^{i \tilde \omega_3}$,
where  $r_j>0$ and $\omega_j, \tilde \omega_3 \in [0,2\pi)$ for $1\leq j\le 3$. Then
$$
V_1(x)= a_1+r_1\cos (2\pi k_1 \cdot x+\omega_1)  + r_2\cos (2\pi k_2 \cdot x+\omega_2) + r_3\cos (2\pi k_3 \cdot x+\omega_3) 
$$
and
$$
V_2(x)= a_1+r_1\cos (2\pi  k_1 \cdot x+\omega_1)  + r_2\cos (2\pi  k_2 \cdot x+\omega_2) + r_3\cos (2\pi  k_3 \cdot x+\tilde \omega_3).
$$
Again by translations, we may further assume that $\omega_1=\omega_2=0$.  
We write $k_3=\alpha_1 k_1+\alpha_2 k_2$ for some $\alpha_1, \alpha_2 \in \mathbb Q$.
Then it is clear from the definition of $M(\cdot)$ that
$$
\max_{\mathbb{T}^2} V_1=a_1+ M(\omega_3)   \quad \mathrm{and} \quad \max_{\mathbb{T}^2} V_2=a_1+M(\tilde \omega_3).
$$
In light of \eqref{min-V}, we get $M(\omega_3)=M(\tilde \omega_3)$.
 Assume that
$$
l=m\pi+m_1\alpha_1\pi +m_2\alpha_2\pi.
$$
for some $m, m_1, m_2\in \mathbb{Z}$.   Accordingly,  by Corollary \ref{key-fact},  we have two cases. 

\medskip

{\bf Case 1.}   $\omega_3=\tilde \omega_3+2kl$ for some $k\in  \mathbb{Z}$. Choose $x_0$ such that  
$$
\begin{cases}
k_1\cdot  x_0=km_1,\\
k_2\cdot x_0=km_2.
\end{cases}
$$
Then $k_3\cdot x_0=k\alpha _1m_1+k \alpha_2 m_2$ and 
$$
V_1(x)=V_2(x+x_0) \quad \text{ for all } x\in \T^2.
$$

\medskip

 {\bf Case 2.}  $\omega_3=2kl-\tilde \omega_3$ for some $k\in  \mathbb{Z}$. Choose $x_0$ such that  
$$
\begin{cases}
k_1\cdot  x_0=km_1,\\
k_2\cdot x_0=km_2.
\end{cases}
$$
Then  $k_3\cdot x_0=k\alpha _1m_1+k \alpha_2 m_2$ and 
$$
V_1(x)=V_2(-x-x_0) \quad \text{ for all } x\in \T^2.
$$

\end{proof}

\begin{rem}\label{rem:higher-coeff}
 It is natural to try using more the coefficients $\{a_j\}_{j \in \N}$ in the asymptotic expansion of $\ol{H}^\ep$ instead of \eqref{min-V} to prove the last step above.
It is, however, quite hard to implement this idea.  
Let us still mention it here.

Choose $(m_1,m_2,m_3)\in  \mathbb{N}^3$ such  that the  $\text{gcd}(m_1,m_2,m_3)=1$ and
$$
m_2k_{2}=m_1k_{1}+m_3k_{3}.
$$
Let $L=m_1+m_2+m_3$.  
It is easy to see that  $a_L$ is the first coefficient that provides us information about $\{\lambda_{1j}\}_{j=1}^3, \{\lambda_{2j}\}_{j=1}^3$
further than Lemma \ref{lem:2d-1}.   
For $r_j=|\lambda_{1j}|=|\lambda_{2j}|$ for $1\leq j \leq 3$, we have
$$
a_L=P(r_j, \ k_j,  \ Q: \ 1\leq j\leq 3)+J(k_1,k_2,k_3, Q)\text{Re}\left({\lambda_{11}^{m_1}} \left(\ol{\lambda_{12}}\right)^{m_2}\lambda_{13}^{m_3}\right).
$$
Here $P$ is a real valued function depending only on $\{r_j, \ k_j,  \ Q: \ 1\leq j\leq 3\}$ and $J$ a real valued function depending only on $\{k_1,k_2,k_3, Q\}$.  
It will be done if we can manage to show that $J(k_1,k_2,k_3, Q)$ is not zero for some  $Q\in \R^2$.  
However, it is not clear to us how to verify  this since the expression of $J$ is too complicated.  
\end{rem}

\section{Proof of Theorem \ref{thm:3d} and \ref{thm:3d-2}} \label{sec:3d}

We first provide the proof of  Theorem \ref{thm:3d}.

\begin{proof}[{\bf Proof of Theorem \ref{thm:3d}}]
We consider each case separately.

\medskip

(1) The sufficiency part follows immediately from Lemma \ref{lem:2d-1}.   Let us prove the converse.  
Since $\{k_{lj}\}_{j=1}^3$ is linearly independent, by suitable translations ($x\mapsto x+x_{0l}$), we may assume that
$$
V_1(x) =\sum_{j=1}^{3} r_j\cos ( 2\pi k_{1j}\cdot x)
$$
and for $c_j>0$,
$$
V_2(x) =\sum_{j=1}^{3} r_j\cos (c_j 2\pi k_{1j}\cdot x).
$$
Then the conclusion follows from Lemma \ref{decomposition} and  changing of variables.

\medskip

(2) Let us first prove the sufficiency part.  
Clearly,  $k_{12}+k_{13}$ and $k_{22}+k_{23}$ are sole vectors.  Since $\{k_{lj}\}_{j=1}^3$ is linearly independent,  due  to  Lemma \ref{lem:2d-1} and Remark \ref{rem:sole},  we must have 
$$
k_{12}+k_{13}\parallel k_{22}+k_{23}.
$$
Hence there exists $c\in \mathbb{Q}$ such that $k_{22}=ck_{12}$ and $k_{23}=ck_{13}$. 

We now prove the converse. By suitable translations, we may assume that 
$$
V_1(x)= r_1\cos (2\pi k_{11}\cdot x)+ r_2\cos (2\pi k_{12}\cdot x)+r_3\cos (2\pi k_{13}\cdot x)
$$
and for $c_1>0$,
$$
V_2(x)=r_1\cos (c_1 2\pi  k_{11}\cdot x)+ r_2\cos (c 2\pi  k_{12}\cdot x)+r_3\cos (c 2\pi k_{13}\cdot x).
$$
We then use Lemma \ref{decomposition} and  changing of variables to get the conclusion.

\medskip

(3)  The necessity part is obvious.  Let us prove  the sufficiency.  Due to  Lemma \ref{lem:2d-1},  there are two cases.

\medskip 

 {\bf Case 1.}  $\{k_{1j}\}_{j=1}^3$ is  linearly independent.  Due to symmetry,  we may assume that $k_{11}$ is not perpendicular to $k_{12}$ and $k_{13}$. Then similar to (2),  we have that 
$$
k_{11}+k_{12}\parallel k_{21}+k_{22}  \quad \mathrm{and}  \quad k_{11}+k_{13}\parallel k_{21}+k_{23}.
$$
Hence there exists $c\in \mathbb{Q}$ such that for $j=1,2,3$, 
$$
k_{2j}=ck_{1j}.
$$
Since $\{k_{1j}\}_{j=1}^3$ is linearly independent, it is easy to see that we can find $x_0\in  \R^n$ such that
$$
V_1(x)=V_2\left({x\over c}+x_0\right) \quad \text{ for all } x\in \T^n.
$$

\medskip

{\bf Case 2.}   $\{k_{1j}\}_{j=1}^3$ is  linearly dependent.  The situation is essentially reduced to the $2$-dimensional case and 
the conclusion follows from Theorem \ref{thm:2d}. 

\end{proof}

\medskip 

Next,  let us prove Theorem \ref{thm:3d-2}.  

\begin{proof}[{\bf Proof of Theorem \ref{thm:3d-2}}]
We consider each situation separately.

\medskip

(1) follows immediately from Lemma  \ref{lem:2d-1}. 

\medskip

(2) The proof of part (i) is similar to (1) of Theorem \ref{thm:3d}, and is omitted.
Let us now consider part (ii). Since $k_{11}$ and $k_{12}$ are linearly independent and non-orthogonal, due  to  Lemma \ref{lem:2d-1} and Remark \ref{rem:sole},  we get that 
$$
k_{11}+k_{12}\parallel k_{21}+k_{22}.
$$
So there exists $c\in \mathbb{Q} \setminus \{0\}$ such that, for $j=1,2$. 
$$
k_{2j}=ck_{1j}.
$$
Accordingly, it is easy to see that we can  find $x_0$ such that 
$$
V_1(x)=V_2\left({x\over c}+x_0\right) \quad \text{ for all } x\in \T^n.
$$

\end{proof}

The following is a simple lemma which should be well known to experts. We leave its proof as an exercise to the interested readers.
 
\begin {lem}\label{decomposition} 
Let $n,n_1,n_2 \in \N$ be such that $n=n_1+n_2$.
For $x\in \R^n$, we write $x=(x_1,x_2,\ldots,x_n)=(x',x'') \in \R^{n_1} \times \R^{n_2}$,
where  $x'=(x_1,x_2,...,x_{n_1})$ and $x''=(x_{n_1+1},...,x_{n})$.
Similarly, for $p\in \R^n$, we write $p=(p',p'') \in \R^{n_1} \times \R^{n_2}$.

Let $W_j\in C(\Bbb T^{n_j})$ be a given potential energy and  $c_j \in \R \setminus \{0\}$ be a given constant for $j=1,2$. 
 Assume that $\ol{H}_1(p'), \ol{H}_2(p'') ,\ol{H}(p)$ are the effective Hamiltonians associated with the Hamiltonians $\frac{1}{2}|p'|^2 +W_1(x')$,
 $\frac{1}{2}|p''|^2 +W_2(x'')$, $\frac{1}{2}|p|^2 + W_1(c_1 x') + W_2(c_2 x'')$, respectively.
Then
$$
\overline H(p)=\overline H_1(p')+\overline H_2(p'') \quad \text{ for all } p=(p',p'') \in \R^{n_1} \times \R^{n_2}.
$$
In particular, $\ol{H}$ is independent of $c_1$ and $c_2$. 
\end{lem}

\bibliographystyle{plain}

\end{document}